\documentclass[11pt]{article}
\usepackage{a4wide}
\usepackage[a4paper, margin=2.3cm]{geometry}
\usepackage{amsmath,amssymb,amsthm}
\usepackage{color}
\usepackage{parskip}
\usepackage{hyperref}
\usepackage{parskip}
\usepackage{setspace}
\usepackage{graphicx}
\usepackage{amsmath}
\usepackage{amssymb}
\usepackage{stackengine}
\usepackage{scalerel,stackengine}
\stackMath
\usepackage{scalerel}

\usepackage{mathtools}
\usepackage[thinc]{esdiff}
 
\usepackage[dvipsnames]{xcolor}

\usepackage{setspace}

\renewcommand{\le}{\leqslant}
\renewcommand{\ge}{\geqslant}
 
\numberwithin{equation}{section}

\newcommand\reallywidehat[1]{%
\savestack{\tmpbox}{\stretchto{%
  \scaleto{%
    \scalerel*[\widthof{\ensuremath{#1}}]{\kern.1pt\mathchar"0362\kern.1pt}%
    {\rule{0ex}{\textheight}}
  }{\textheight}%
}{2.4ex}}%
\stackon[-6.9pt]{#1}{\tmpbox}%
}
\theoremstyle{plain}
\newtheorem{theorem}{Theorem}[section]
\newtheorem*{theorem*}{Theorem}
\newtheorem{lemma}[theorem]{Lemma}
\newtheorem{corollary}[theorem]{Corollary}
\newtheorem{conjecture}[theorem]{Conjecture}

\theoremstyle{remark}

\theoremstyle{definition}

\newtheorem{remark}[theorem]{Remark}

\newcommand{\comment}[1]{}

\theoremstyle{remark}

\theoremstyle{definition}

\DeclareMathOperator{\supp}{supp}

\newcommand{\be}{\begin{equation}}
\newcommand{\ee}{\end{equation}}
\newcommand{\bee}{\begin{equation*}}
\newcommand{\eee}{\end{equation*}}

\def\supp{\hbox{supp\,}}

\def\steven#1{\noindent
\textcolor{magenta}
{\textsc{(Steven:}
\textsf{#1})}}

\def\supp{\hbox{supp\,}}
\allowdisplaybreaks

\begin{document}
\title{On the solvability of systems of equations revisited}
\author{Thang Pham\thanks{University of Science, Vietnam National University, Hanoi. Email: thangpham.math@vnu.edu.vn} ~~~~~~~~~~\and Steven Senger\thanks{Department of Mathematics, Missouri State University. Email: StevenSenger@MissouriState.edu} ~~~~~~~~~~\and Nguyen Trung-Tuan\thanks{University of Science, Vietnam National University, Hanoi. Email: tuan.nguyentrung@gmail.com}\and Nguyen Duc-Thang\thanks{University of Science, Vietnam National University, HCM. Email: ndthanghcmus@gmail.com} \and Le Anh Vinh\thanks{Vietnam Institute of Educational Sciences. Email: vinhle@vnies.edu.vn}}
\maketitle
\begin{abstract}
In this paper, we introduce a new and direct approach to study the solvability of systems of equations generated by bilinear forms. More precisely, let $B (\cdot,
  \cdot)$ be a non-degenerate bilinear form and $E$ be a set in $\mathbb{F}_q^2$. We prove that if $|E|\gg q^{5/3}$ then the number of triples $(B(x, y), B(y, z), B(z, x))$ with $x, y, z\in E$ is at least $cq^3$ for some positive constant $c$. This significantly improves a result due to the fifth listed author (2009). 
\end{abstract}
\section{Introduction}
Let $\mathbb{F}_q$ be a finite field of order $q$, where $q$ is a prime power. Let $B (\cdot,
  \cdot)$ be a non-degenerate bilinear form in $\mathbb{F}_q^d$. Let $1\le t\le k-1$ and $1\le l\le \binom{k}{2}$ be integers. In this paper, we study the solvability of the system of $l$ bilinear equations with $k$ variables $a_1, \ldots, a_k,$
  \begin{equation}\label{mainsystem}B(a_i, a_j)=\lambda_{ij},~1\le i, j\le k,~\lambda_{ij}\in \mathbb{F}_q,\end{equation}
  where each variable appears in at most $t \leq k - 1$ equations.

When $l=1$, Hart and Iosevich \cite{HI}, using discrete Fourier analysis, proved that if $E_1, E_2\subset \mathbb{F}_q^d$ with $|E_1||E_2|\gg q^{d+1}$ then, for any $\lambda\ne 0$, the equation $B(a_1, a_2)=\lambda$ has  $(1+o(1))|E_1||E_2|/q$ solutions $(a_1, a_2)\in E_1\times E_2$. Hart, Iosevich, Koh and Rudnev \cite{HI2} provided a sharp example for this result, namely, for any $\epsilon>0$,
there are sets $E\subset \mathbb{F}_q^d$ with $|E|\sim q^{\frac{d+1}{2}-\epsilon}$ and $|\{x\cdot y\colon x, y\in A\}|=o(q)$.

When $l>1$ and $E_1, \ldots, E_k$ are sets in $\mathbb{F}_q^d$, Vinh \cite{Vinh1}, using graph theoretic methods, proved that if  $|E_i | \gg q^{\frac{d - 1}{2} + t}$ for all $i$, then, for any tuple $(\lambda_{ij})$ with entries in $\mathbb{F}_q^*$, the above system  has
  \[ (1 + o (1)) q^{- l} \prod_{i = 1}^k |E_i | \]
  solutions $(a_1, \ldots, a_k)\in E_1\times \cdots\times E_k$.

If $l=\binom{k}{2}$, $t=k-1$, and $E_1=\cdots=E_k=E$, in another paper \cite{Vinh2} with a refined argument, he showed that under a weaker condition that $|E|\gg q^{\frac{d-1}{2}+\frac{k}{2}}$, the system (\ref{mainsystem}) is solvable for $(1 - o (1))
  q^{\binom{k}{2}}$ possible choices of tuples $(\lambda_{i j})$ in $(\mathbb{F}_q^*)^{\binom{k}{2}}$. In the case $d=2$, $k=3$, and $l=\binom{3}{2}$, he showed that the system is solvable for at least $\gg \left(\frac{|E|}{q^2}\right)\cdot q^3$ triples $(\lambda_1, \lambda_2, \lambda_3)\in (\mathbb{F}_q^*)^3$. Thus, if one wishes to have a lower bound of $\gg q^3$, then the condition $|E|\gg q^2$ is needed. For this type problem, the level of difficulty is increasing when the dimension is getting smaller. Our paper is devoted to study the two dimensional case. Our main theorem is stated as follows. 
\begin{theorem}\label{mainTriangles}
Let $E\subset \mathbb{F}_q^2$. Assume $|E|\gg q^{5/3}$, then 
\[|\{(B(x, y), B(y, z), B(z, x))\colon x, y, z\in E\}|\gg q^3.\]
\end{theorem}
From the sharpness construction in \cite{HI2} for the case $k=2$, we are led to the following conjecture. 
\begin{conjecture}
   Let $E\subset \mathbb{F}_q^2$. Assume $|E|\gg q^{3/2}$, then 
\[|\{(B(x, y), B(y, z), B(z, x))\colon x, y, z\in E\}|\gg q^3.\]
\end{conjecture}
Extend the method to higher dimensions, we obtain the following. 

\begin{theorem}\label{higher-dimension}
    Let $E\subset \mathbb{F}_q^d$. Assume that $|E|\gg q^{d-\frac{d-1}{3}}$, then 
    \[|\{(B(x, y), B(y, z), B(z, x))\colon x, y, z\in E\}|\gg q^3.\]
\end{theorem}

\paragraph{Discussions on related results:} There are several results in the literature when the bilinear form is replaced by other functions. If we define $F(a_i, a_j):=||a_i-a_j||=(a_i^1-a_j^1)^2+\cdots+(a_i^d-a_j^d)^2$, then the solvability of the system \ref{mainsystem} has been studied intensively. Bennett, Hart, Iosevich, Pakianathan, and Rudnev \cite{BHI} established that if $k\le d$, $l=\binom{k}{2}$, and $t=k-1$ then the condition $|E|\gg q^{d-\frac{d-1}{k}}$ is sufficient to conclude that the system (\ref{mainsystem}) is solvable for $(1 - o (1))
  q^{\binom{k}{2}}$ possible choices of tuples $(\lambda_{i j})$ in $(\mathbb{F}_q^*)^{\binom{k}{2}}$. In two dimensions, they also showed that the exponent $5/3$ can be decreased to $8/5$. 

\paragraph{Discussions on the technique:} It is not possible to use the methods developed in \cite{BHI, Vinh2} to prove Theorem \ref{mainTriangles}. Let us try to explain in detail. For a function $F$, the method used in \cite{BHI} is adaptable when the following is satisfied: $F(a, b)=F(u, v)$ if and only if there exists a group $G$ such that $a=gu$ and $b=gv$ for some $g\in G$. For example, if $F(x, y)=x\cdot y^\perp$, then $G$ can be chosen as the group $SL_2(\mathbb{F}_q)$. However, in the case of the dot product function, we cannot find such a group. One more remark we want to add here is that the set $\{(x\cdot y, y\cdot z, z\cdot x)\colon x, y, z\in E\}$ is not the same as the set $\{(x\cdot y^\perp, y\cdot z^\perp, z\cdot x^\perp)\colon x, y, z\in E\}$ in general, and we do not know if there exists a map $\psi\colon \mathbb{F}_q^3\to \mathbb{F}_q^3$ such that it is (almost) size-invariant between these two sets. 

We now discuss the method developed in \cite{Vinh2}. To prove Theorem \ref{mainTriangles} with this approach, the most challenging problem is to show that the number of tuples $(a_1, a_2, a_3, a_4)\in E^4$ such that $B(a_1, a_3)=B(a_1, a_4)=B(a_2, a_3)=B(a_2, a_4)=1$ (cycles of length four) is at most $|E|^4/q^4$ whenever $|E|\gg q^{3/2}$. Assume $B(a_i, a_j)=a_i\cdot a_j$ for simplicity, then it is not difficult to see that for each such tuple we have either $a_1=a_2$ or $a_3=a_4$, otherwise, there are two lines intersect at two points. However, we know that the number of such tuples is $(1+o(1))|E|^3/q^2$ under $|E|\gg q^{3/2}$ (see \cite{mot, hai}), and $|E|^3/q^2$ is smaller than $|E|^4/q^4$ only if $|E|>q^2$. 

In this paper, we introduce a direct approach to bound the $L^2$-norm of \textit{the counting function} from above, and it is worth noting that our argument also gives the same exponent for other functions, for example, diagonal polynomials and norm functions, which are not possible if we apply the methods in \cite{BHI, Vinh2}. This raises a question of clarifying all functions $F$ such that Theorem \ref{higher-dimension} holds when $B(\cdot, \cdot)$ is replaced by $F(\cdot, \cdot)$. When $k\ge 4$, there are some technical issues that will be discussed in detail in the last section.

\paragraph{Some remarks in the continuous setting:}
In the continuous setting, Eswarathasan, Iosevich, and Taylor \cite{EIT} proved that for a compact set $E\subset \mathbb{R}^d$ if the Hausdorff dimension of $E$ is at least $\frac{d+1}{2}$, then the set $\{B(x, y)\colon x, y\in E\}$ has positive Lebesgue measure. This is directly in line with the finite field analog of Hart and Iosevich in \cite{HI}. Based on Theorem \ref{mainTriangles}, it might be reasonable to make the conjecture that for a compact set $E\subset \mathbb{R}^d$ if the Hausdorff dimension of $E$ is at least $d-\frac{d-1}{k}$, then the Lebesgue measure of the set $\{(B(x_i, x_j))_{1\le i<j\le k}\colon x_i, x_j\in E\}$ is positive. So, in two dimensions with $k=3$, the dimensional threshold is expected to be $5/3$. We plan to address this conjecture in a subsequent paper. 

\paragraph{Notations:} Throughout this paper, by $X\gg Y$, we mean there exists a positive constant $C$, independent of $q$, such that $X\ge CY$, $X=o(Y)$ if $X/Y\to 0$ as $q\to \infty$, and $X\sim Y$ if $X\ll Y\ll X$.
\section{Proof of Theorem \ref{mainTriangles}}
In our approach, there is no harm to assume that $B(x, y)=x\cdot y$. 

We begin by stating the main lemma and using it to prove Theorem \ref{mainTriangles}. Let $E_1, E_2,$ and $E_3$ be three subsets of $E$ with the properties that $|E_i|\gg |E|$ and there is no line passing through the origin and intersecting any pair of the $E_i$. That is, each of the $E_i$ is supported on a set of lines through the origin that is disjoint from $E_j$ for $j\neq i$. 

We now state and prove an $L^2$-estimate specifically for triangles. Define a triangle counting function
\[f(\lambda_1,\lambda_2,\lambda_3):=|\{(x,y,z)\in E_1\times E_2\times E_3:x\cdot y = \lambda_1, y\cdot z = \lambda_2, z \cdot x = \lambda_3\}|,\]
for $\lambda_1, \lambda_2, \lambda_3\in \mathbb{F}_q$.

\begin{lemma}\label{L2bound}
    Given $E\subset \mathbb F_q^2,$ with $|E|\gg q^\frac{5}{3},$ partitioned into three disjoint subsets, $E_1, E_2,$ and $E_3,$ each with size $|E_j|\gg|E|,$ each supported on distinct lines through the origin, we have that
    \[\sum_{\lambda_1, \lambda_2, \lambda_3\in\mathbb F_q}f^2(\lambda_1, \lambda_2, \lambda_3)\ll \frac{|E|^6}{q^3}+q^2|E|^3+|E|^4.\]
\end{lemma}

The proof of this lemma is given in Section \ref{sbb} below. For now, we quickly prove Theorem \ref{mainTriangles} using Lemma \ref{L2bound}. We use the notation $1_{\lambda_1, \lambda_2, \lambda_3}$ to denote the function that is one when $f(\lambda_1, \lambda_2, \lambda_3)>0$ and zero otherwise. To do this, we first appeal to the size conditions on the sets $E_i,$ then apply the Cauchy-Schwarz inequality

\begin{align*}
|E|^6&\ll \left(\sum_{\lambda_1, \lambda_2, \lambda_3\in\mathbb F_q}f(\lambda_1, \lambda_2,\lambda_3)\right)^2\leq \left(\sum_{\lambda_1, \lambda_2, \lambda_3\in\mathbb F_q}1_{\lambda_1,\lambda_2,\lambda_3}^2\right)\left(\sum_{\lambda_1, \lambda_2, \lambda_3\in\mathbb F_q}f^2(\lambda_1, \lambda_2,\lambda_3)\right)\\
&\ll |\supp(f)|\cdot\left( \frac{|E|^6}{q^3} +q^2|E|^3+ |E|^4\right).
\end{align*}
Since $|E|\gg q^{5/3}$, we obtain
\[|\supp(f)|\gg\frac{|E|^6}{q^{-3}|E|^6} \gg q^3,\]
which completes the proof.

\section{Proof of the main lemma (Lemma \ref{L2bound})}\label{sbb}
We first recall some results in the literature about the case $k=2$.
\begin{lemma}[\cite{HI}]\label{cor231}
Let $E$ and $F$ be two sets in $\mathbb{F}_q^d$ and let $\lambda\in \mathbb{F}_q^*$. Let $N_\lambda(E, F)$ be the number of pairs $(x, y)\in E\times F$ such that $x\cdot y=\lambda$. 
Then we have 
\[\left\vert N_{\lambda}(E, F)-\frac{|E||F|}{q}\right\vert \ll q^{\frac{d-1}{2}}|E|^{1/2}|F|^{1/2}.\]
\end{lemma}
\begin{lemma}[\cite{HI}]\label{coco}
Let $E$ and $F$ be two sets in $\mathbb{F}_q^d$. Let $N$ be the number of tuples $(x_1, x_2, y_1, y_2)\in E\times E\times F\times F$ such that $x_1\cdot y_1=x_2\cdot y_2$. Then
\[\left\vert N-\frac{|E|^2|F|^2}{q}\right\vert \ll q^d|E||F|.\]
\end{lemma}
The following corollary is straightforward.
\begin{corollary}[\cite{HI}]\label{cor23}
Let $E$ and $F$ be two sets in $\mathbb{F}_q^d$ and let $\chi$ be an additive character of $\mathbb{F}_q$. Then we have 
\[\left\vert \sum_{x_1, x_2\in E}\sum_{y_1, y_2\in F}\sum_{t\in \mathbb{F}_q}\chi(t\cdot (x_1\cdot y_1-x_2\cdot y_2))-|E|^2|F|^2\right\vert \ll q^{d+1}|E||F|.\]
\end{corollary}
We also make use of a result on the distribution of paths in a given set in $\mathbb{F}_q^2$.
\begin{lemma}[\cite{mot}]\label{star}
Let $E$ be a set in $\mathbb{F}_q^d$ and let $\lambda, \beta\in \mathbb{F}_q^*$. Assume that $|E|\gg q^{\frac{d+1}{2}}$, then the number of triples $(x, y, z)\in E^3$ such that $x\cdot y=\lambda$ and $x\cdot z=\beta$ is at most $|E|^3/q^2$.
\end{lemma}

With these lemmas in hand, we are ready to complete the proof of Lemma \ref{L2bound}.

\begin{proof}[Proof of Lemma \ref{L2bound}]
    We count triangles whose points come from the distinct sets, $E_j,$ using character sums. Let $\chi$ be an additive character of $\mathbb{F}_q$. We write
    \[f(\lambda_1, \lambda_2, \lambda_3) = q^{-3}\sum_{x\in E_1,~y\in E_2,~z\in E_3}\sum_{s_1,s_2, s_3\in \mathbb F_q}\chi(s_1(x\cdot y-\lambda_1))\chi(s_2(y\cdot z-\lambda_2))\chi(s_3(z\cdot x-\lambda_3)).\]
    Using this expression, we can similarly write the $L^2$ sum as
    \begin{align*}
    \sum_{\lambda_1, \lambda_2, \lambda_3\in \mathbb F_q} f^2(\lambda_1, \lambda_2, \lambda_3) &= \sum_{\lambda_1, \lambda_2, \lambda_3\in \mathbb F_q}q^{-6}\sum_{\substack{x\in E_1,~y\in E_2,~z\in E_3\\x'\in E_1,~y'\in E_2,~z'\in E_3}}\sum_{\substack{s_1,s_2, s_3\\s_1',s_2',s_3'}}\Big(\chi(s_1(\lambda_1-x\cdot y))\chi(s_2(\lambda_2-y\cdot z))\\
    &\chi(s_3(\lambda_3 - z\cdot x))\chi(s_1'(x'\cdot y'-\lambda_1))\chi(s_2'(y'\cdot z'-\lambda_2))\chi(s_3'(z'\cdot x'-\lambda_3))\Big)\\
    &= \sum_{\lambda_1, \lambda_2, \lambda_3\in \mathbb F_q}q^{-6}\sum_{\substack{x\in E_1,~y\in E_2,~z\in E_3\\x'\in E_1,~y'\in E_2,~z'\in E_3}}\sum_{\substack{s_1,s_2, s_3\\s_1',s_2',s_3'}}\Big(\chi(\lambda_1(s_1-s_1'))\chi(\lambda_2(s_2-s_2'))\\
    &\chi(\lambda_3(s_3-s_3'))\chi(s_1 x\cdot y-s_1' x'\cdot y')\chi(s_2y\cdot z-s_2'y'\cdot z')\chi(s_3 z\cdot x-s_3' z' \cdot x')\Big).
    \end{align*}
    By orthogonality, any term where there is an $s_i\neq s_i'$ will return a zero, so it is enough to consider the terms where $s_i=s_i'.$ We next separate out the main term, which is the term where all of the $s_i$ are zero, as we can compute it directly. This means that we can sum through the $q^3$ triples of $(\lambda_1, \lambda_2, \lambda_3)$ and get
    \[
    = q^3\left(\frac{|E|^6}{q^6}+q^{-6}\hskip -2em\sum_{\substack{x\in E_1,~y\in E_2,~z\in E_3\\x'\in E_1,~y'\in E_2,~z'\in E_3}}\sum_{\substack{(s_1, s_2, s_3)\\ \neq (0,0,0)}}
    \chi(s_1 (x\cdot y- x'\cdot y'))\chi(s_2(y\cdot z-y'\cdot z'))\chi(s_3( z\cdot x- z' \cdot x')\right).
    \]
    This is equal to
    \[
    \frac{|E|^6}{q^3}+q^{-3}\hskip -2em\sum_{\substack{x\in E_1,~y\in E_2,~z\in E_3\\x'\in E_1,~y'\in E_2,~z'\in E_3}}\sum_{\substack{(s_1, s_2, s_3)\\ \neq (0,0,0)}}
    \chi(s_1 (x\cdot y- x'\cdot y'))\chi(s_2(y\cdot z-y'\cdot z'))\chi(s_3( z\cdot x- z' \cdot x')).
    \]
    With the main term computed, we now turn our attention to the rest of the sum. We further split this sum up into the cases where we have $s_1$ is nonzero while $s_2$ and $s_3$ range through all of $\mathbb F_q$, which we denote by $I$, and when we have the terms where $s_1=0$ and $(s_2, s_3)\ne (0, 0)$, which we denote by $II$. For simplicity, we define the set of points that give dot product $\lambda_i$ with a point $x$ by $\ell_{\lambda_i}(x).$
    \begin{align*}
    I&=q^{-3}\sum_{\substack{x\in E_1,~y\in E_2,~z\in E_3\\x'\in E_1,~y'\in E_2,~z'\in E_3}}\sum_{\substack{s_1\neq 0\\ s_2,s_3\in\mathbb F_q}}
    \chi(s_1 (x\cdot y- x'\cdot y'))\chi(s_2(y\cdot z-y'\cdot z'))\chi(s_3( z\cdot x- z' \cdot x')\\
    &=q^{-3}\sum_{\substack{x\in E_1,~y\in E_2\\x'\in E_2,~y'\in E_2}}\sum_{s_1\neq 0}
    \chi(s_1 (x\cdot y- x'\cdot y'))\sum_{z,z'\in E_3}\sum_{s_2,s_2\in\mathbb F_q}\chi(s_2(y\cdot z-y'\cdot z'))\chi(s_3( z\cdot x- z' \cdot x').
    \end{align*}
    We now bound the inner sum by the following observation. If $y\cdot z\ne y'\cdot z'$ or $z\cdot x\ne z'\cdot x'$, then one of the sums over $s_2$ or $s_3$ is zero. Thus, we only need to consider tuples with $y\cdot z= y'\cdot z'$ and $z\cdot x= z'\cdot x'$. We notice that for fixed choices of $x, x' \in E_1$ and $y, y' \in E_2,$ and $z \in E_3,$ we have a very restrictive incidence structure. Namely, the point $z$ must lie on $\ell_\lambda(x)$ and $\ell_\mu(y)$ for some $\lambda, \mu \in \mathbb F_q$, fixing their values. Then, with the values of $\lambda$ and $\mu$ fixed, the point $z'$ has only one possible choice; it must lie on $\ell_\lambda(x')\cap \ell_\mu(y').$ Note that $|\ell_\lambda(x')\cap \ell_\mu(y')|\le 1$ if $x$ and $y$ belong to different sets $E_i$. This explains why we partitioned the set $E$ into $E_1, E_2,$ and $E_3$ at the beginning. This property will also be used several times in the rest of the proof. Therefore, since there are only $|E_3|$ choices for $z$, and each value of $z$ determines the only one possibility of $z',$ we can estimate this inner sum above by $|E_3|$ for any quadruple of $x, x'\in E_1,$ and $y,y'\in E_2.$ This gives us
    \begin{align*}
    I&=q^{-3}\sum_{\substack{x\in E_1,~y\in E_2\\x'\in E_2,~y'\in E_2}}\sum_{s_1\neq 0}
    \chi(s_1 (x\cdot y- x'\cdot y'))\sum_{z,z'\in E_3}\sum_{s_2,s_2\in\mathbb F_q}\chi(s_2(y\cdot z-y'\cdot z'))\chi(s_3( z\cdot x- z' \cdot x')\\
    &\leq q^{-1}\sum_{\substack{x\in E_1,~y\in E_2\\x'\in E_2,~y'\in E_2}}\sum_{s_1\neq 0}
    \chi(s_1 (x\cdot y- x'\cdot y'))\cdot |E_3|.\\
    \end{align*}
    We now add and subtract the term where $s_1=0$, remember the sizes of $E_1, E_2,$ and $E_3$, then consider the rest of the sum as another $L^2$-type sum counting dot products to get
    \begin{align*}
    I&\leq q^{-1}|E_3|\sum_{\substack{x\in E_1,~y\in E_2\\x'\in E_2,~y'\in E_2}}\sum_{s_1\neq 0}
    \chi(s_1 (x\cdot y- x'\cdot y'))\\
    &= q^{-1}|E_3|\left(\sum_{\substack{x\in E_1,~y\in E_2\\x'\in E_2,~y'\in E_2}}\sum_{s_1\in \mathbb F_q}
    \chi(s_1 (x\cdot y- x'\cdot y'))-|E_1|^2|E_2|^2\right).
    \end{align*}
By Corollary \ref{cor23} and the fact that $|E_i|\sim |E|$, we obtain $I\ll q^2|E|^3.$

We now turn our attention to estimating $II,$ the case where $s_1=0.$ Recall the definition of $II,$ keeping in mind that $s_1=0$
    \[
    II = q^{-3}\sum_{\substack{x\in E_1,~y\in E_2,~z\in E_3\\x'\in E_1,~y'\in E_2,~z'\in E_3}}\sum_{\substack{(s_2, s_3)\\ \neq (0,0)}}
    \chi(s_1 (x\cdot y- x'\cdot y'))\chi(s_2(y\cdot z-y'\cdot z'))\chi(s_3( z\cdot x- z' \cdot x').
    \]
    
    Because we have $s_1=0$, so the factor of $\chi(s_1(x\cdot y - x'\cdot y'))$ just becomes $\chi(0)=1.$ Also, we have that $(s_2, s_3)\ne (0, 0)$, so we break this sum up into terms where $s_2=0$ and terms where $s_2\neq 0.$ Notice that the terms where $s_2=0$ must have $s_3\neq 0,$ and the factors of $\chi(s_2(y\cdot z - y'\cdot z')$ will become one. Hence,
    \begin{align*}
    II&=q^{-3}\sum_{\substack{x\in E_1,~y\in E_2,~z\in E_3\\x'\in E_1,~y'\in E_2,~z'\in E_3}}\sum_{(s_2,s_3)\neq (0,0)}\chi(s_2(y\cdot z-y'\cdot z'))\chi(s_3( z\cdot x- z' \cdot x')\\
    &=q^{-3}\sum_{\substack{x\in E_1,~y\in E_2,~z\in E_3\\x'\in E_1,~y'\in E_2,~z'\in E_3}}\sum_{s_3\neq 0}\chi(s_3( z\cdot x- z' \cdot x'))\\
    &+q^{-3}\sum_{\substack{x\in E_1,~y\in E_2,~z\in E_3\\x'\in E_1,~y'\in E_2,~z'\in E_3}}\sum_{s_2\neq 0}\chi(s_2(y\cdot z-y'\cdot z'))\sum_{s_3\in \mathbb F_q}\chi(s_3( z\cdot x- z' \cdot x')\\
    &=:II_1+II_2.
    \end{align*}
Regarding $II_1$, one has
    \begin{align*}
        II_1 &= q^{-3}|E_2|^2\sum_{\substack{x\in E_1,~z\in E_3\\x'\in E_1,~z'\in E_3}}\sum_{s_3\neq 0}\chi(s_3( z\cdot x- z' \cdot x'))\\
        &= q^{-3}|E_2|^2\left(\sum_{z\cdot x = z' \cdot x'}\sum_{s_3\neq 0}\chi(s_3( z\cdot x- z' \cdot x'))+\sum_{z \cdot x \neq z' \cdot x'}\sum_{s_3 \neq 0}\chi(s_3( z\cdot x- z' \cdot x'))\right).
    \end{align*}
    By orthogonality, the second sum in this expression must be negative. We continue, using this fact and the size estimates on the $E_i$, to get
    \begin{align*}
        II_1 &\leq q^{-3}|E_2|^2\sum_{x\cdot z = x' \cdot z'}\sum_{s_3 \neq 0}\chi(s_3( z\cdot x- z' \cdot x'))\\
        &=q^{-3}|E_2|^2(q-1)|\{(x,z,x',z')\in E_1\times E_3 \times E_1 \times E_3 : z\cdot x = z' \cdot x'\}|\\
        &\ll q^{-2}|E|^2|\{(x,z,x',z')\in E_1\times E_3 \times E_1 \times E_3 : z\cdot x = z' \cdot x'\}|.
    \end{align*}
From here, Lemma \ref{coco} gives $II_1\ll q^{-3}|E|^6$ when $|E|\gg q^{3/2}$.

Regarding $II_2,$ we begin by noticing that if $z\cdot x \neq z' \cdot x',$ then the inner sum will give zero for that term, so moving forward, we can assume that $z\cdot x = z' \cdot x'.$ So
    \begin{align*}
        II_2&=q^{-3}\sum_{\substack{x\in E_1,~y\in E_2,~z\in E_3\\x'\in E_1,~y'\in E_2,~z'\in E_3}}\sum_{s_2\neq 0}\chi(s_2(y\cdot z-y'\cdot z'))\sum_{s_3\in \mathbb F_q}\chi(s_3( z\cdot x- z' \cdot x'))\\
        &=q^{-2}\sum_{\substack{y,y'\in E_2,\\ z, z'\in E_3\\
        z\cdot x = z'\cdot x'}}\sum_{s_2\neq 0}\chi(s_2(y\cdot z-y'\cdot z')).\\
    \end{align*}
    We then notice that if $y\cdot z \neq y'\cdot z',$ then the sum will be negative, by orthogonality of $\chi$. So for an upper bound, it suffices to also assume that $y\cdot z = y' \cdot z',$ giving
    \begin{align*}
        II_2&=q^{-2}\sum_{\substack{y,y'\in E_2,\\z\cdot x = z'\cdot x'}}\sum_{s_2\neq 0}\chi(s_2(y\cdot z-y'\cdot z'))\\
        &\leq q^{-1}\sum_{\substack{y\cdot z = y' \cdot z',\\z\cdot x = z'\cdot x'}}1\\
        &= q^{-1}|\{(x,y,z,x',y',z')\in (E_1\times E_2\times E_3)^2: y\cdot z = y' \cdot z',z\cdot x = z'\cdot x'\}|\\
        &=q^{-1}\left(|A| + |B| + |C|\right),
    \end{align*}
    where we split the set into three disjoint subsets, described below. For any triple in the set $(x,y,z)\in E_1\times E_2\times E_3$, there are fixed values $t= z \cdot x$ and $u= y \cdot z.$ Let $A$ be the part of the set with both $t$ and $u$ nonzero, $B$ be the part with one of $t$ and $u$ zero, and the other nonzero, and $C$ be the part of the set with $t=u=0.$ 
    
First we treat $A$, the case that $t$ and $u$ are both nonzero. Since $t$ and $u$ are both nonzero, the number of triples $(x',y',z')$ satisfying $z'\cdot x' = t$ and $y' \cdot z' = u$, under $|E|\gg q^{3/2}$, is at most $q^{-2}|E|^3$ by Lemma \ref{star}. Hence, $|A|\ll q^{-2}|E|^6.$ 

Now we consider $B$, the case that one of $t$ and $u$ is zero. Without loss of generality, we suppose $t=0$ and $u\neq 0.$ To estimate this case, we first estimate the number of quadruples $(y,z,y',z')\in (E_2\times E_3)^2$ satisfying $y\cdot z = u = y' \cdot z'.$ This is at most $(q^{-1}|E|^2)^2$ by Lemma \ref{cor231}. Once this quadruple has been fixed, the choices of $z, z'\in E_3,$ give at most $q$ choices for each of $x\in\ell_0(z)$ and $x'\in\ell_0(z'),$ respectively. Summarizing over all $u\ne 0$ gives $|B|\ll (q^{-1}|E|^2)^2\cdot q^2\cdot q=q|E|^4$.

Finally, we consider the subset $C,$ where $t=u=0.$ In this case, we first choose $x,y,x'y'\in(E_1\times E_2)^2,$ and notice that because $E_1$ and $E_2$ are supported on different lines through the origin, there is at most one choice of $z\in \ell_0(x) \cap \ell_0(y)$ for any given pair $(x,y)\in E_1\times E_2$ as well as at most one choice of $z' \in \ell_0(x')\cap \ell_0(y')$ for any given pair $(x',y')\in E_1\times E_2,$ giving an upper bound of $|E_1|^2|E_2|^2$ for this portion of the set. Recalling the estimates on the sizes of the $E_i$ leaves us with an upper bound of $|C| \ll |E|^4.$ 

Combining the bounds on the subsets $A, B$, and $C$, we see that
    \[II_2\ll q^{-1}\left(|A|+|B|+|C|\right)\ll q^{-1}\left(q^{-2}|E|^6+ q|E|^4+ |E|^4\right)\ll q^{-3}|E|^6,\]
since $|E|\gg q^{3/2}$.

In other words, $I+II\ll q^{-3}|E|^6$ whenever $|E|\gg q^{5/3}$, which completes the proof.
\end{proof}
\section{Proof of Theorem \ref{higher-dimension}}
The idea is the same as in the plane case. We will also need to bound $I$, $II_1$, and $II_2$. 

As mentioned in the introduction, Theorem \ref{higher-dimension} follows if $|E|\gg q^{\frac{d-1+k}{2}}$. Thus, without loss of generality, we suppose from now that $|E|\ll q^{\frac{d-1+k}{2}}$. The same argument implies $I\ll q^{2d-2}|E|^3.$
Lemma \ref{coco} also gives $II_1\ll q^{-3}|E|^6$ when $|E|\gg q^{\frac{d+1}{2}}$. Bounding $II_2$, we repeat the argument for $A$, $B$, and $C$. 

For $|A|$, we obtain the same upper bound of $q^{-2}|E|^6$ by Lemma \ref{star}. 

For $|B|$, we obtain 
\[|B|\ll \left(\frac{|E|^2}{q}\right)^2\cdot q^{2d-2}\cdot q=q^{2d-3}|E|^4.\]
We note that in order to obtain the exponent $d-\frac{d-1}{3}$ at the end, this upper bound of $B$ is not good enough. To be precise, we want $\frac{|E|^6}{q^3}\gg \frac{|B|}{q}$, so we would need the condition that $|E|^2\gg q^{2d-1}$, or $|E|\gg q^{d-\frac{1}{2}}$, which is worse than expected. The same happens for $|C|$. To proceed further, we need to consider the following two cases:

{\bf Case 1: $d=3$}

We count the number of tuples $(x, y, z, x', y', z')\in E^6$ such that $y\cdot z=y'\cdot z'$ and $x\cdot z=x'\cdot z'=0$. To bound the number of such tuples, we define $g(z)=|\{x\in E\colon x\cdot z=0\}|$ for $z\in E$ and zero otherwise, and $h$ to be the indicator function of $E$. With this setting, one has 
\[|B|+|C|\le  \sum_{\substack{y, z, y', z'\\ y\cdot z=y'\cdot z'}}g(z)g(z')h(y)h(y').\]
To tackle this functional sum, we need generalized versions of lemmas in Section 3. 

First the proof of Lemma \ref{coco} in \cite{HI} can be written in the form of complex-valued functions.

\begin{lemma}[\cite{HI}]\label{cococo}
Let $g$ and $h$ be complex-valued functions from $\mathbb{F}_q^d$ to $\mathbb{C}$. We have 
\[\left\vert \sum_{\substack{x_1, x_2, y_1, y_2\in \mathbb{F}_q^d\\ x_1\cdot y_1=x_2\cdot y_2}}g(x_1)g(x_2)h(y_1)h(y_2)-\frac{||g||_1^2||h||_1^2}{q} \right\vert\ll q^{d}||g||_2^2||h||_2^2,\]
here 
\[||g||_p^p=\sum_{x\in \mathbb{F}_q^d}|g(x)|^p.\]
\end{lemma}

Second we need a result on the distribution of pairs of zero dot-product. 
\begin{lemma}[\cite{HI2}]\label{zero-product}
    For $E, F\subset \mathbb{F}_q^d$. The number of pairs $(x, y)\in E\times F$ such that $x\cdot y=0$ is at most 
    \[\frac{|E||F|}{q}+q^{\frac{d}{2}}\sqrt{|E||F|}.\]
\end{lemma}

We now use Lemma \ref{cococo} and Lemma \ref{zero-product} to bound $|B|+|C|$. Indeed, by definition of $g$, we know that 
\[||g||_1\ll q^{\frac{3}{2}}|E|,\]
since we assume at the beginning that $|E|\ll q^{\frac{5}{2}}$. To bound $||g||_2^2$ from above, we observe that it is at most the number of triples $(x, y, z)\in E^3$ such that $z\cdot x=z\cdot y=0$. Since any plane in $\mathbb{F}_q^3$ contains at most $q^2$ points, we have 
\[||g||_2^2\ll q^2\cdot ||g||_1\ll q^{\frac{7}{2}}|E|.\]
So, $|B|+|C|\ll q^{3+\frac{7}{2}}|E|^2$.

In other words, 
\[II_2\ll \frac{|E|^6}{q^3}+q^{\frac{11}{2}}|E|^2.\]
Putting all computations together, we obtain 
\[I+II\ll \frac{|E|^6}{q^3}+q^{\frac{11}{2}}|E|^2+q^{2d-2}|E|^3.\]
This completes the proof of this case. 

{\bf Case 2: $d\ge 4$}

Notice that under $d\ge 4$, we have $d-\frac{d-1}{3}\ge \frac{d+2}{2}$. Thus, by Lemma \ref{zero-product}, the number of pairs $(x, y)\in E\times E$ such that $x\cdot y=0$ is at most $2|E|^2/q$. Thus, we always can choose a subset $E'\subset E$ such that $|E'|\gg |E|$ and for each $x\in E'$, the number of $y\in E$ with $x\cdot y=0$ is at most $C|E|/q$ for some large constant $C>0$. In this case, at the beginning, we should work with $E'$ instead of $E$. So, by abuse of notation, we might assume that for each $x\in E$, the number of $y\in E$ such that $x\cdot y=0$ is at most $C|E|/q$. With this property, by Lemma \ref{coco}, one has 
\[|B|+|C|\ll \left(\frac{|E|}{q}\right)^2\cdot \left(\frac{|E|^4}{q}+q^d|E|^2\right)\ll \frac{|E|^6}{q^3},\]
under $|E|\gg q^{\frac{2d+1}{3}}$.

In other words,
\[I+II\ll \frac{|E|^6}{q^3}+q^{2d-2}|E|^3.\]
This completes the proof of this case. 
\section{Discussions in higher dimensions}
We now discuss if the proof of Theorem \ref{higher-dimension} can be extended to the case $k>3$. 

Let's assume $k=4$ and $d=3$. By repeating the argument, one of the sums will be
 \begin{align*}
    I=&q^{-6}\sum_{\substack{x\in E_1,~y\in E_2,~z\in E_3\\x'\in E_1,~y'\in E_2,~z'\in E_3\\w\in E_4, w'\in E_4}}\sum_{\substack{(s_1, s_2, s_3)\ne (0, 0, 0)\\ s_4, s_5, s_6\in \mathbb{F}_q}}
    \chi(s_1 (x\cdot y- x'\cdot y'))\chi(s_2(y\cdot z-y'\cdot z'))\chi(s_3( z\cdot x- z' \cdot x'))\\
    &\cdot \chi(s_4(w\cdot x-w'\cdot x'))\chi(s_5(w\cdot y-w'\cdot y'))\chi(s_6(w\cdot z-w'\cdot z')).\\
    \end{align*}
Observe that if 
\[(x\cdot y-x'\cdot y', y\cdot z-y'\cdot z', z\cdot x-z'\cdot x')\ne (0, 0, 0),\]
then 
\[\sum_{(s_1, s_2, s_3)\ne (0, 0, 0)} \chi(s_1 (x\cdot y- x'\cdot y'))\chi(s_2(y\cdot z-y'\cdot z'))\chi(s_3( z\cdot x- z' \cdot x'))<0,\]
by orthogonality. Hence, in the above sum, we can assume that 
\[x\cdot y=x'\cdot y', ~y\cdot z=y'\cdot z', ~z\cdot x=z'\cdot x'.\]
The proof of Theorem \ref{higher-dimension} tells us that the number of such tuples $(x, y, z, x', y', z')$ is at most $q^{-3}|E|^6+q^{4}|E|^3$. 

Moreover, if one of the terms $w\cdot x-w'\cdot x'$, $w\cdot y-w'\cdot y'$, and $w\cdot z-w'\cdot z'$ is non-zero, then $I=0$ by the orthogonality of $\chi$. Thus, we can assume in addition that all of them are zero. 

If either $x$, $y$, and $z$ or $x'$, $y'$, and $z'$ are independent, then $I$ can be bounded by $q^4|E|^4$. Indeed, assume $x'$, $y'$, and $z'$ are linearly independent, then for each fixed $w\in E_4$, the number of $w'\in E_4$ such that 
\[ w\cdot x-w'\cdot x'=w\cdot y-w'\cdot y'=w\cdot z-w'\cdot z'=0\]
is at most one. So, in total, $I\ll q^4|E|^4$. Compared to the main term $|E|^8/q^6$, we would need the condition $|E|\gg q^{5/2}$ at the end. This is directly in line with the result for the distance function mentioned in the introduction.

However, in general, we cannot partition the set $E$ into three sets $E_1$, $E_2$, and $E_3$ such that $|E_i|\sim |E|$ and any triple $(x, y, z)\in E_1\times E_2\times E_3$ forms a linearly independent system. The following presents such an obstacle in three dimensions.

\begin{lemma}\label{obstacle}
    Given $E\subset \mathbb{F}_q^3$ with $|E|=q^{2+\epsilon}$ for some $\epsilon>0$. It is impossible to find $E_1, E_2, E_3\subset E$ such that $|E_i|\sim |E|$ and all triples $(x, y, z)\in E_1\times E_2\times E_3$ are linearly independent. 
\end{lemma}
\begin{proof}[Proof of Lemma \ref{obstacle}]
    Suppose there exist such three sets. This means that any line passing through the origin intersects only one of those three sets. A $2$-dimensional subspace is called spanned by $E_1\times E_2$ if it intersects both $E_1$ and $E_2$. It is clear that such a plane does not intersect $E_3$. We denote the set of those planes by $S$. We note that the total number of $2$-dimensional subspaces in $\mathbb{F}_q^3$ is $\sim q^2$. We now show that the size of $S$ is about $q^2$. Indeed, let $S_1$ be the set of subspaces avoiding $E_1$ and $S_2$ be the set of subspaces avoiding $E_2$. Then it is clear that 
   \[I(S_1, E_1)=0=I(S_2, E_2).\]
    Note that Vinh's theorem \cite{SZT} tells us that 
    \[\left\vert I(S_i, E_i)-\frac{|E_i||S_i|}{q}\right\vert \ll q\sqrt{|S_i||E_i|}.\]
    This implies $|S_i|\ll q^{2-\epsilon}$. 
In other words, most subspaces are spanned by $E_1$ and $E_2$, i.e. $|S|\gg q^2$. We apply the incidence bound above again, then it is clear that $I(S, E_3)>0$. So we have a contradiction. 
\end{proof}

This technical step might different when considering other functions. In the case of the distance function, we need to partition $E$ into $E_1$, $E_2$, and $E_3$ with $|E_i|\sim |E|$, and any line in $\mathbb{F}_q^3$ intersects at most two of these sets. If $|E|=q^{2+\epsilon}$, with the same proof and incidence bound between points and lines in \cite{PPV}, we can prove that it is impossible. However, one can think of sets of very different sizes, which might help to overcome this difficulty. 

In conclusion, when $d\ge 3$ and $k\ge 4$, it is still an open question of proving the exponent $d-\frac{d-1}{k}$ for the system of bilinear equations with $k$ variables and $l=\binom{k}{2}$. It also would be interesting to see if the method in this paper can be adapted in other settings, say, quasi-fields or finite p-adic rings. 

\section{Acknowledgements}
We would like to thank to the VIASM for the hospitality and for the excellent working
condition. This research was supported by Vietnam National Foundation for Science and Technology Development (NAFOSTED) under grant number 101.99--2021.09.

\end{document}